\documentclass[12pt]{amsart}
\usepackage{SSdefn}
\usepackage{eucal}

\DeclareMathOperator{\ini}{in}

\title{Gr\"obner-coherent rings and modules}
\date{\today}

\author{Rohit Nagpal}
\address{Department of Mathematics, University of Chicago, Chicago, IL}
\email{\href{mailto:nagpal@math.uchicago.edu}{nagpal@math.uchicago.edu}}
\urladdr{\url{http://math.uchicago.edu/~nagpal/}}

\author{Andrew Snowden}
\address{Department of Mathematics, University of Michigan, Ann Arbor, MI}
\email{\href{mailto:asnowden@umich.edu}{asnowden@umich.edu}}
\urladdr{\url{http://www-personal.umich.edu/~asnowden/}}

\thanks{AS was supported by NSF grants DMS-1303082 and DMS-1453893.}

\begin{document}

\begin{abstract}
Let $R$ be a graded ring. We introduce a class of graded $R$-modules called Gr\"obner-coherent modules. Roughly, these are graded $R$-modules that are coherent as ungraded modules because they admit an adequate theory of Gr\"obner bases. The class of Gr\"obner-coherent modules is formally similar to the class of coherent modules: for instance, it is an abelian category closed under extension. However, Gr\"obner-coherent modules come with tools for effective computation that are not present for coherent modules.
\end{abstract}

\maketitle
\tableofcontents

\section{Introduction}

Let $R$ be a graded ring. The purpose of this paper is to isolate a particular class of graded $R$-modules: the Gr\"obner-coherent modules. Roughly speaking, these are graded $R$-modules that are coherent as ungraded $R$-modules but not simply by happenstance: they have a ``reason'' for their coherence related to the grading (that reason basically being an adequate theory of Gr\"obner bases). Formally, the class of Gr\"obner-coherent modules over $R$ behaves similarly to the class of coherent modules: for instance, it is an abelian category and closed under extensions. However, Gr\"obner-coherent modules enjoy an advantage over coherent modules in that they come with effective computation procedures. For instance, if $R$ is a Gr\"obner-coherent ring (meaning Gr\"obner-coherent as a module over itself) and one can compute with finitely presented graded $R$-modules then one can also compute with finitely presented (ungraded) $R$-modules.

Our original motivation for developing the theory of Gr\"obner-coherence was in relation to our study of divided power algebras \cite{dp}. Let $\bD$ be the divided power algebra in one variable over a noetherian ring $\bk$. Typically, $\bD$ is not noetherian. We show that $\bD$ is a Gr\"obner-coherent ring. As a consequence, if one can compute with finitely generated $\bk$-modules then one can also compute with finitely presented $\bD$-modules.

\section{Background on coherence}

Let $R$ be a ring. An $R$-module $M$ is {\bf coherent} if it is finitely generated and every finitely generated submodule is finitely presented; clearly, a coherent module is finitely presented. The category of coherent modules is abelian and closed under extensions. The ring $R$ is {\bf coherent} if it is coherent as a module over itself. In this case, every finitely presented module is coherent. Whenever we refer to coherence of a graded ring or module we ignore the grading.

Suppose now that $R$ is a graded ring. Then a graded $R$-module $M$ is {\bf graded-coherent} if it is finitely generated and every finitely generated homogeneous submodule is finitely presented. Once again, a graded-coherent module is finitely presented, and the category of graded-coherent modules is abelian and closed under extensions. The ring $R$ is {\bf graded-coherent} if it is graded-coherent as a module, and then every finitely presented graded module is graded-coherent.

%

\section{Gr\"obner bases}

Gr\"obner bases are typically employed to study ideals in a polynomial ring, or, more generally, submodules of free modules over polynomial rings. However, the ideas apply equally well to study inhomogeneous submodules of graded modules over an arbitrary graded ring. In this section, we develop the theory of Gr\"obner bases in this greater generality.

Let $R$ be a graded ring, and let $M$ be a graded $R$-module. We assume $R$ is supported in non-negative degrees and that $M_n=0$ for $n \ll 0$. Let $x=\sum_{n \in \bZ} x_n$ be a nonzero element of $M$ with $x_n \in M_n$. We define the {\bf degree} of $x$, denoted $\deg(x)$, to be the maximal $n$ such that $x_n \ne 0$. We define the {\bf initial term} of $x$, denoted $\ini(x)$, to be $x_n$ where $n=\deg(x)$. Given an (inhomogeneous) submodule $N$ of $M$, we define the {\bf initial submodule}, denoted $\ini(N)$, to be the homogeneous submodule generated by $\ini(x)$ over all nonzero $x \in N$.

\begin{definition}
Let $N$ be a submodule of a graded module $M$. A collection $\{x_i\}_{i \in I}$ of elements in $N$ is a {\bf Gr\"obner basis} for $N$ if the $\ini(x_i)$ generate $\ini(N)$ as an $R$-module.
\end{definition}

\begin{definition}
Let $N$ be a submodule of a graded module $M$, let $\{x_i\}_{i \in I}$ be a collection of elements of $N$, and let $y$ be another element of $N$. An expression  $y=\sum_{i \in I} a_i x_i$ with $a_i \in R$ is {\bf reduced} if $\deg(a_i)+\deg(x_i) \le \deg(y)$ for all $i$.
\end{definition}

\begin{proposition}
\label{prop:reduced-expression}
Let $M$ be a graded $R$-module, let $N$ be a submodule, and let $\{x_i\}_{i \in I}$ be a collection of elements of $N$. Let $\sum_{i \in I} c_{i,j} \ini(x_i)=0$, for $j$ in an index set $J$, be a homogeneous generating set for the module of syzygies of the $\ini(x_i)$, and put $z_j = \sum_{i \in I} c_{i,j} x_i$. Then the following are equivalent:
\begin{enumerate}
\item The $x_i$ form a Gr\"obner basis for $N$.
\item Every element of $N$ has a reduced expression in terms of the $x_i$.
\item The $x_i$ generate $N$ and each $z_j$ has a reduced expression in terms of the $x_i$.
\end{enumerate}
\end{proposition}

\begin{proof}
(a) $\implies$ (b). Let $\rS_n$ be the statement ``every element of $N$ of degree $\le n$ has a reduced expression in terms of the $x_i$.'' For $n \ll 0$, the statement $\rS_n$ is obviously true. Now suppose that $\rS_{n-1}$ is true, and let us prove $\rS_n$. Thus let $y \in N$ be an element of degree $n$. We then have an expression $\ini(y) = \sum_{i \in I} a_i \ini(x_i)$ where $\deg(a_i)+\deg(x_i)=\deg(y)$. Let $y'=y-\sum_{i \in I} a_i \ini(x_i)$. Then $y' \in N$ and $\deg(y') \le n-1$, so by $\rS_{n-1}$ we have an expression $y' = \sum b_i x_i$ with $\deg(b_i)+\deg(x_i) \le n-1$. Thus $y=\sum (a_i+b_i) x_i$ is reduced expression for $y$, and $\rS_n$ follows. We conclude that $\rS_n$ holds for all $n$ by induction, and so (b) holds.

(b) $\implies$ (a). Let $y$ be a nonzero element of $N$, and let us show that $\ini(y)$ can be generated by the $\ini(x_i)$. Using (b), we have an expression $y=\sum_{i \in I} a_i x_i$ where $\deg(a_i)+\deg(x_i) \le \deg(y)$. Let $J \subset I$ be the set of indices $i$ such that $a_i x_i$ has degree equal to $\deg(y)$. Then $\ini(y) = \sum_{i \in J} \ini(a_i x_i) = \sum_{i \in J} \ini(a_i) \ini(x_i)$. Thus the $\{x_i\}$ form a Gr\"obner basis, and so (a) holds.

(b) $\implies$ (c) is immediate, (b) implies that the $x_i$ generate $N$, and explicitly states that the $z_j$ admit reduced expressions.

(c) $\implies$ (b). Let $y \in N$ be given. Let $y = \sum_{i \in I} a_i x_i$ be an expression for $y$ in terms of the $x_i$'s with $\delta=\max_{i \in I}(\deg(a_i)+\deg(x_i))$ minimal. We claim $\delta=\deg(y)$, and so the expression is reduced. Assume not. Write $a_i=a_i'+a_i''$ where $a_i''$ is the degree $\delta-\deg(x_i)$ piece of $a_i$ (and $a_i'$ has smaller degree). Then $\sum_{i \in I} a_i'' \ini(x_i)=0$. This is a homogeneous syzygy of the $\ini(x_i)$, and so we have an expression in terms of the $c$'s: there exist homogeneous elements $b_j \in R$ satisfying $\deg(b_j)+\deg(c_{i,j})= \deg(a''_i)$ such that $a''_i = \sum_{j \in J} b_j c_{i,j}$. We have
\begin{displaymath}
y=\sum_{i \in I} a_i x_i = \sum_{i \in I} a_i' x_i + \sum_{i \in I} \sum_{j \in J} b_j c_{i,j} x_i
= \sum_{i \in I} a_i' x_i + \sum_{j \in J} b_j z_j.
\end{displaymath}
Using the reduced expression for the $z$'s, this gives an expression $y=\sum_{i \in I} d_i x_i$ with $\max(\deg(d_i)+\deg(x_i))<\delta$, a contradiction. Thus $\delta=\deg(y)$ as claimed, and (b) holds.
\end{proof}

The above proposition leads to Buchberger's algorithm for finding a Gr\"obner basis. Let $X=\{x_i\}_{i \in I}$ be a generating set for $N \subset M$. The algorithm proceeds as follows:
\begin{enumerate}
\item Compute the set $Z=\{z_j\}_{j \in J}$ as in the proposition.
\item If each $z_j$ has a reduced expression, output $X$ and terminate.
\item Otherwise, replace $X$ with $X \cup Z$ and return to step (a).
\end{enumerate}
If the algorithm terminates, then its output is a Gr\"obner basis. If the input set $X$ is finite and $M$ is graded-coherent, then $X$ will remain finite after each step, and so if the algorithm terminates it will produce a finite Gr\"obner basis. Note that in the usual description of Buchberger's algorithm, one does not add the $z_j$'s to $X$, but their remainders after applying the generalized division algorithm. This is more efficient, but makes no theoretical difference.

In general, analysis of Buchberger's algorithm can be quite difficult. However, there is one situation that we can analyze easily:

\begin{proposition}
\label{prop:buchberger}
Let $M$ be a graded module, $N$ a submodule, and $\{x_i\}_{i \in I}$ a generating set for $N$. Suppose that there exists an integer $\delta \ge 0$ such that every element $y \in N$ admits an expression of the form $y=\sum_{i \in I} a_i x_i$ with $\deg(a_i)+\deg(x_i) \le \deg(y)+\delta$. Then Buchberger's algorithm terminates after at most $\delta$ steps.
\end{proposition}

\begin{proof}
This is like (c) $\implies$ (b) from the previous proposition: each step in the algorithm lets us reduce $\delta$ by at least one, so we eventually get down to $\delta=0$.
\end{proof}

The following proposition shows how Gr\"obner bases can be used to compute syzygies.

\begin{proposition} \label{grob-syz}
Let $M$ be a graded module, let $N$ be a submodule, and let $\{x_i\}_{i \in I}$ be a Gr\"obner basis for $N$. Let $\sum_{i \in I} c_{i,j} \ini(x_i)=0$, for $j \in J$, be a homogeneous generating set for the module of syzygies of the $\ini(x_i)$, let $z_j=\sum_{i \in I} c_{i,j} x_i$, and let $z_j=\sum_{i \in I} d_{i,j} x_i$ be a reduced expression. Let $\Phi \colon \bigoplus_{i \in I} R e_i \to N$ be the map defined by $\Phi(e_i)=x_i$. We regard $e_i$ as homogeneous of degree $\deg(x_i)$. Let $r_j=\sum_{i \in I} (c_{i,j}-d_{i,j}) e_i$. Then $\{r_j\}_{j \in J}$ is a Gr\"obner basis for $K=\ker(\Phi)$.
\end{proposition}

\begin{proof}
By definition, $\deg(c_{i,j})+\deg(x_i)=d_j$ is independent of $i$. We have $\deg(z_j)<d_j$, and so $\deg(d_{i,j})+\deg(x_i)<d_j$ as well. Thus $\ini(r_j)=\sum_{i \in I} c_{i,j} e_i$. Now, suppose that $s=\sum_{i \in I} a_i e_i$ is an element of $K$. Thus $\sum_{i \in I} a_i x_i = 0$. Let $\delta = \deg(s) = \max_{i \in I}(\deg(a_i)+\deg(x_i))$, and write $a_i=a_i'+a_i''$ where $a_i''$ is the homogeneous degree $\delta-\deg(x_i)$ piece of $a_i$. Then $\sum_{i \in I} a_i'' \ini(x_i)=0$ is a homogeneous syzygy, and so there are homogeneous element $b_j \in R$ with $\deg(b_j)+\deg(c_{i,j})=\deg(a_i'')$ such that $a_i''=\sum_{j \in J} c_{i,j} b_j$. We have
\begin{displaymath}
s-\sum_{j \in J} b_j r_j=\sum_{i \in I} a_i' e_i+ \sum_{i \in I} \sum_{j \in J} b_j d_{i,j} e_i.
\end{displaymath}
Note that $\deg(a_i') + \deg(e_i) = \deg(a_i') + \deg(x_i) < \delta$ and $\deg(b_j) + \deg(d_{i,j}) + \deg(e_i) < \deg(b_j) + d_j = \deg(b_j) + \deg(c_{i,j}) + \deg(x_i) \le \delta$. Thus the equation above shows that $\deg(s-\sum_{j \in J} b_j r_j) < \delta$ and so $\ini(s)$ is generated by the $\ini(r_j)$. This implies that $\{r_j\}_{j\in J}$ is a Gr\"obner basis for $K$.
\end{proof}

%

\section{Gr\"obner coherence}

The primary definitions are:

\begin{definition} \label{defn:grobcoh}
Let $R$ be a graded ring. We say that a graded $R$-module $M$ is {\bf Gr\"obner-coherent} if it is graded-coherent and every finitely generated inhomogeneous submodule admits a finite Gr\"obner basis.
\end{definition}

\begin{definition}
We say that a graded ring $R$ is {\bf Gr\"obner-coherent} if it is so as a module over itself.
\end{definition}

\begin{remark}
The two conditions in Definition~\ref{defn:grobcoh} (namely, graded-coherent and every submodule admits a finite Gr\"obner basis) play off of each other nicely, as computations with Gr\"obner bases often reduce to computations with leading terms, and the graded-coherence ensures that such computations behave well.
\end{remark}

\begin{proposition}
Let $M$ be a Gr\"obner-coherent graded module. Then $M$ is coherent.
\end{proposition}

\begin{proof}
Let $N$ be a finitely generated inhomogeneous submodule of $M$. Let $\{x_i\}_{i \in I}$ be a finite Gr\"obner basis for $N$, and let $\Phi \colon \bigoplus_{i \in I} Re_i \to N$ be the surjection defined by $\Phi(e_i)=x_i$. It follows from Proposition~\ref{grob-syz}, that $\ker(\Phi)$ has a finite Gr\"obner basis (note that the graded-coherence of $M$ implies that the set $J$ in Proposition~\ref{grob-syz} is finite), and so $\ker(\Phi)$ is finitely generated. This shows that $N$ is finitely presented, completing the proof.
\end{proof}

\begin{proposition} \label{grobserre}
Let $M$ be a Gr\"obner-coherent $R$-module and let $M' \subset M$ be a finitely generated homogeneous submodule. Then $M'$ and $M/M'$ are both Gr\"obner-coherent.
\end{proposition}

\begin{proof}
It follows from basic properties of graded-coherence that $M'$ and $M/M'$ are graded-coherent. If $N$ is a finitely generated submodule of $M'$, then it is also one of $M$, and thus admits a finite Gr\"obner basis. Thus $M'$ is Gr\"obner-coherent.

Now let $N$ be a finitely generated submodule of $M/M'$, and let $\wt{N}$ be its inverse image in $M$, which is finitely generated. Let $\{\wt{x_i}\}_{i \in I}$ be a finite Gr\"obner basis for $\wt{N}$, and let $x_i$ be the image of $\wt{x}_i$ in $N$. Let $y \in N$ and let $\wt{y}$ be a lift of $y$ to $\wt{N}$ with $\deg(y)=\deg(\wt{y})$. Let $\wt{y}=\sum_{i \in I} a_i \wt{x}_i$ be a reduced expression. Then $y=\sum_{i \in I} a_i x_i$ is also a reduced expression, and so $\{x_i\}_{i \in I}$ is a finite Gr\"obner basis for $N$. This shows that $M/M'$ is Gr\"obner-coherent.
\end{proof}

\begin{proposition} \label{grobserre2}
Let $M' \subset M$ be graded $R$-modules such that $M'$ and $M/M'$ are Gr\"obner-coherent. Then $M$ is Gr\"obner-coherent.
\end{proposition}

\begin{proof}
It follows from basic properties of graded-coherence that $M$ is graded-coherent. Let $N \subset M$ be a finitely generated submodule, and let $\ol{N}$ be its image in $M/M'$. Let $\{\ol{x}_i\}_{i \in I}$ be a finite Gr\"obner basis for $\ol{N}$, and let $x_i \in N$ be a lift of $\ol{x}_i$. Note that we cannot necessarily pick $x_i$ to have the same degree as $\ol{x}_i$. Let $\delta=\max_{i \in I}(\deg(x_i)-\deg(\ol{x}_i))$. Since $\ol{N}$ is coherent, the kernel $K$ of the map $N \to \ol{N}$ is a finitely generated submodule of $M'$. Let $\{x'_j\}_{j \in J}$ be a Gr\"obner basis for $K$.

Now let $y$ be an element of $N$, and let $\ol{y}$ be its image in $\ol{N}$. Let $\ol{y}=\sum_{i \in I} a_i \ol{x}_i$ be a reduced expression, so that $\deg(a_i)+\deg(\ol{x}_i) \le \deg(\ol{y}) \le \deg(y)$. We have $\deg(a_i)+\deg(x_i) \le \deg(y)+\delta$. Put $y'=y-\sum_{i \in I} a_i x_i$. This is an element of $K$ satisfying $\deg(y') \le \deg(y)+\delta$. Let $y'=\sum_{j \in J} a_i' x_j'$ be a reduced expression, so that $\deg(a_i')+\deg(x_j') \le \deg(y)+\delta$. We thus have an expression $y=\sum_{i \in I} a_i x_i+\sum_{j \in J} a_j' x_j'$ where $\deg(a_i)+\deg(x_i) \le \deg(y)+\delta$ and $\deg(a_j')+\deg(x_j') \le \deg(y)+\delta$. This expression shows that the $x_i$ and $x_j'$ generate $N$, and it follows from Proposition~\ref{prop:buchberger} that Buchberger's algorithm applied to this generating set stops after at most $\delta$ steps, producing a finite Gr\"obner basis for $N$. Thus $M$ is Gr\"obner-coherent.
\end{proof}

\begin{corollary} \label{cor:grobserre2}
A finite direct sum of Gr\"obner-coherent modules is Gr\"obner-coherent.
\end{corollary}

\begin{corollary}
The category of Gr\"obner-coherent $R$-modules is an abelian subcategory of the category of all graded $R$-modules, and is closed under extension.
\end{corollary}
\begin{proof} Suppose $f \colon M \to N$ is a map of Gr\"obner-coherent modules. Since the category of graded-coherent modules is abelian,  we know that kernel, cokernel, and image of $f$ are graded-coherent. In particular, these objects are finitely generated. Hence $\ker(f)$, $\coker(f)$ and $\im(f)$ are Gr\"obner-coherent by Proposition~\ref{grobserre}. The statement about extensions follows from Proposition~\ref{grobserre2}.
\end{proof}

\begin{proposition}
Let $R$ be Gr\"obner-coherent and let $M$ be a graded $R$-module. Then $M$ is Gr\"obner-coherent if and only if $M$ is finitely presented.
\end{proposition}

\begin{proof}
Suppose that $M$ is finitely presented, and write $M=F/K$ where $F$ is a finite rank free module and $K$ is a finitely generated submodule. Then $F$, being a sum of shifts of $R$, is Gr\"obner-coherent by Corollary~\ref{cor:grobserre2}, and so $M$ is Gr\"obner-coherent by Proposition~\ref{grobserre}. Conversely, a Gr\"obner-coherent module (over any ring) is coherent, and thus finitely presented.
\end{proof}

%

The following proposition shows that one can effectively compute a Gr\"obner basis for an inhomogeneous submodule of a Gr\"obner-coherent module using Buchberger's algorithm, starting from any set of generators.

\begin{proposition}
Let $M$ be a Gr\"obner-coherence $R$-module, let $N$ be a finitely generated inhomogeneous submodule of $M$, and let $x_1, \ldots, x_n$ be a set of generators for $N$. Then Buchburger's algorithm applied to $x_1, \ldots, x_n$ terminates after finitely many steps and yields a Gr\"obner basis for $N$.
\end{proposition}

\begin{proof}
Since $M$ is Gr\"obner-coherent, $N$ admits a finite Gr\"obner basis, say $y_1, \ldots, y_m$. Since the $x_i$ generate $N$, we can write $y_i=\sum_{j=1}^n a_{i,j} x_j$ for scalars $a_{i,j} \in R$. Let $\delta$ be the maximum value of $\deg(a_{i,j})+\deg(x_j)-\deg(y_i)$ over all $i$ and $j$. Now, let $z \in N$ be an arbitrary element. Since $\{y_i\}$ forms a Gr\"obner basis, we have a reduced expression $z=\sum_{i=1}^n b_i y_i$. This gives $z=\sum_{j=1}^m c_j x_j$ with $c_j=\sum_{i=1}^n a_{i,j} b_i$. We have
\begin{displaymath}
\deg(c_j)+\deg(x_j) \le \deg(a_{i,j})+\deg(b_i)+\deg(x_j) \le \deg(z)+\delta.
\end{displaymath}
Thus Buchberger's algorithm applied to $x_1, \ldots, x_n$ terminates after at most $\delta$ steps by Proposition~\ref{prop:buchberger}.
\end{proof}

\section{Further results}

The following result gives a potentially useful way of establishing Gr\"obner-coherence: once graded-coherence is known, Gr\"obner-coherence is, in a sense, local.

\begin{proposition}
\label{prop:locally-grobner}
Let $R$ be a graded ring and let $\bk=R_0$. Let $M$ be a graded-coherent $R$-module. Then $M$ is Gr\"obner-coherent if and only if $M_{\fm}$ is Gr\"obner-coherent for each maximal ideal $\fm$ of $\bk$.
\end{proposition}

\begin{proof}
First, suppose $M$ is Gr\"obner-coherent and let $\fm$ be a maximal ideal of $\bk$. Let $N$ be a finitely generated submodule of $M_{\fm}$. Then there is a finitely generated module $N'$ of $M$ such that $N = N'_{\fm}$. Let $\{x_i\}_{i \in I}$ be a finite Gr\"obner basis of $N'$ and let $y_i$ be the image of $x_i$  under the natural map $f \colon N' \to N'_{\fm} = N$. Let $y \in \ini(N)$. Then there is an element $x \in N'$ of the same degree as $y$ such that $y = f(x)$. Let $x = \sum_{i \in I}a_i x_i$ be a reduced expression. Then $y = \sum_{i \in I}f(a_i) y_i$ is a reduced expression as well. This shows that $\{y_i\}_{i \in I}$ is a Gr\"obner basis of $N$. 

Conversely, suppose $M_{\fm}$ is Gr\"obner-coherent for each maximal ideal $\fm$ of $\bk$. Let $N$ be a finitely generated inhomogeneous submodule. Suppose that $\{y_i\}_{i \in I}$ is a finite Gr\"obner basis for $N_{\fm}$ for some maximal ideal $\fm$. Multiplying by an element $s \in \fm$ if necessary, we may assume that each $y_i$ are of the form $f(x_i)$ for some $x_i \in N$ where $f$ is the localization map $N \to N_{\fm}$, and by adding in finitely many elements we may assume that $\{x_i\}_{i \in I}$ generate $N$. Let $\{z_j\}_{j \in J}$ be as in Proposition~\ref{prop:reduced-expression}. Since $M$ is graded-coherent $J$ can be assumed to be finite. Each $z_j$ admits a reduced expression in terms of the $x_i$ after localizing at $\fm$. Since there are only finitely many scalars involved, it follows that $z_j$ admits such an expression after
 inverting a single element $s \not\in \fm$. Thus by Proposition~\ref{prop:reduced-expression}, the image of $\{x_i\}_{i \in I}$ under the localization $f_s \colon N \to N[1/s]$ define a Gr\"obner basis for $N[1/s]$. 
 
 Since $M_{\fm}$ is Gr\"obner-coherent for each $\fm$, the above argument shows that there are $s_1, \ldots, s_n$ generating the unit ideal and sets $\{x^k_i\}_{i \in I_k} \subset N$ for $1 \le k \le n$  such that for each $k$ the set $\{f_{s_k}(x^k_i)\}_{i \in I_k}$ is a finite Gr\"obner basis of $N[1/s_k]$ and  $\{x^k_i\}_{i \in I_k}$ generates $N$. We claim that $\bigcup_{1 \le k \le n} \{x^k_i\}_{i \in I_k}$ is a Gr\"obner basis of $N$. To see this let $y \in N$. Multiplying a reduced expression of $f_{s_k}(y)$ in terms of $f_{s_k}(x_i^k)$ by a large enough power $s_k^{n_k}$ we can obtain a reduced expression $\sum_{i \in  I^k} a_i^k x_i^k$ for  $s_k^{n_k}y$. Since  $s_1, \ldots, s_n$ generate the unit ideal there are $c_k$ satisfying $\sum_{k=1}^n c_k s_k^{n_k} = 1$. Now  $\sum_{k=1}^n\sum_{i \in  I^k} c_k a_i^k x_i^k$ is a reduced expression for $y$. This shows that $\cup_{1 \le k \le n} \{x^k_i\}_{i \in I_k}$ is a Gr\"obner basis of $N$, completing the proof.
\end{proof}

The next results show that Gr\"obner properties behave well along flat maps.

\begin{proposition} \label{flatgrob}
Let $R \to S$ be a flat map of graded rings, let $M$ be a graded $R$-module, let $N \subset M$ be an inhomogeneous submodule, and let $\{x_i\}_{i \in I}$ be a Gr\"obner basis for $M$. Then $\{1 \otimes x_i\}_{i\in I}$ is a Gr\"obner basis for $S \otimes_R N \subset S \otimes_R M$.
\end{proposition}

\begin{proof}
Let $\sum_{i \in I} c_{i,j} \ini(x_i)=0$, for $j$ in an index set $J$, be a homogeneous generating set for the module of syzygies of the $\ini(x_i)$. This syzygy module is simply the kernel $K$ of the map $R^{\oplus I} \to M$ sending the $i$th basis vector to $\ini(x_i)$. Since $R \to S$ is flat, it follows that the map $S^{\oplus I} \to S \otimes_R M$ has kernel $S \otimes_R K$. Thus the relations $\sum_{i \in I} c_{i,j} (1 \otimes \ini(x_i))=0$, for $j \in J$, form a homogeneous generating set for the module of syzygies of the $1 \otimes \ini(x_i)$. Let $z_j$ be as in Proposition~\ref{prop:reduced-expression}. Since the $x_i$ form a Gr\"obner basis, each $z_j$ admits a reduced expression in terms of the $x_i$. It follows that $1 \otimes z_j$ also admits a reduced expression in terms of the $1 \otimes x_i$, and thus the $1 \otimes x_i$ form a Gr\"obner basis by Proposition~\ref{prop:reduced-expression}.
\end{proof}

\begin{proposition}
\label{prop:flat-limit-grobner}
Let $\{R_i\}_{i \in I}$ be a directed system of Gr\"obner-coherent graded rings such that for all $i \le j$ the transition map $R_i \to R_j$ is flat. Then the direct limit $R$ is Gr\"obner-coherent.
\end{proposition}

\begin{proof}
It is a standard fact that $R$ is graded-coherent (\cite[Proposition~20]{soublin}). Now let $\fa$ be a finitely generated ideal of $R$. Since $\fa$ is finitely generated, there is some $i \in I$ such that $\fa$ is the extension of an ideal $\fa' \subset R_i$ along the map $R_i \to R$. Since this map is flat, we have $\fa = \fa' \otimes_{R_i} R$, and so a Gr\"obner basis of  $\fa'$ gives one of $\fa$ by Proposition~\ref{flatgrob}. Thus $\fa$ has a finite Gr\"obner basis, and so $R$ is Gr\"obner-coherent.
\end{proof}

\begin{example}
A direct limit of noetherian graded rings with flat transition maps is Gr\"obner-coherent. In particular, a polynomial ring (with any cardinality of variables) over a noetherian coefficient ring is Gr\"obner-coherent.
\end{example}

\section{Relation between different notions of coherence} \label{ss:cohreln}

As we have seen, the following implications hold:
\begin{displaymath}
\textrm{Gr\"obner-coherent} \implies \textrm{coherent} \implies \textrm{graded-coherent}
\end{displaymath}
We now show that both implications are strict.

First, let $\bk$ be a coherent ring such that $\bk[x]$ is not coherent. By \cite[Proposition~18]{soublin}, such a ring exists: in fact, $\bk$ can be taken to be a countable direct product of $\bQ\lbb t,u \rbb$'s. It is easy to show that $\bk[x]$ is graded-coherent (this also follows from \cite[\S 4.5]{dp}). Thus $\bk[x]$ is an example of a ring that is graded-coherent but not coherent.

Next, let $\bk$ be a valuation ring with non-archimedean valuation group. Thus, letting $v$ be the valuation on $\bk$, there exist non-zero $a, b \in \bk$ with $v(a) > 0$ and $v(b) > n v(a)$ for all $n \in \bN$. We claim that $\bk[x]$ is coherent but not Gr\"obner-coherent. The coherence of $\bk[x]$ follows from \cite[pg.~25]{valuation} (or see \cite[Theorem~7.3.3]{sarah}). Now consider the ideal  $I = (ax + 1, b)$ in $\bk[x]$. It is easy to see that the degree zero piece of $\ini(I)$ is the $\bk$-ideal generated by elements of the form $b/a^n$ for $n \ge 0$, and is clearly not finitely generated. Hence $\ini(I)$ is not finitely generated, completing the proof of the claim. (This example comes from \cite[pg.~10]{yengui}.)

\end{document}